\documentclass[11pt,a4paper,reqno]{amsart}

\usepackage [english]{babel}  %sillabazione in inglese

\DeclareFontFamily{U}{mathb}{\hyphenchar\font45}
\DeclareFontShape{U}{mathb}{m}{n}{
      <5> <6> <7> <8> <9> <10>
      <10.95> <12> <14.4> <17.28> <20.74> <24.88>
      mathb10
      }{}
\DeclareSymbolFont{mathb}{U}{mathb}{m}{n}

\DeclareMathSymbol{\sqbullet}{1}{mathb}{"0D}

\usepackage{amsmath,amssymb}
\usepackage{mathtools} % Per fare \bigcup con due righe più allineato
\usepackage{enumitem}
\usepackage{euscript}
\usepackage{graphicx}
\usepackage[all]{xy}
\usepackage[dvipsnames]{xcolor}
\usepackage[colorlinks=true,linktocpage=true,pagebackref=false, citecolor=black,linkcolor=black]{hyperref} %Cambiare false con true per vedere a che pagina sono citate le referenze
\usepackage{pifont}
\usepackage{tikz,pgfplots}
\usetikzlibrary{calc}
\usepackage{pgffor}
\usepackage{enumitem} %%%%%%%% Pacchetto per enumerate

\pgfplotsset{compat=1.17}
\usepgfplotslibrary{fillbetween}
\usetikzlibrary{cd,arrows,decorations.pathmorphing,backgrounds,automata,positioning,fit,matrix}
\usepackage{caption,subcaption}
\usepackage{comment}

\allowdisplaybreaks

\pgfplotsset{soldot/.style={color=black,only marks,mark=*}} \pgfplotsset{holdot/.style={color=black,fill=white,only marks,mark=*}}

\newtheorem{thm}{Theorem}[section]

\newtheorem{lem}[thm]{Lemma}

\theoremstyle{definition}

\newtheorem{example}[thm]{Example}

\theoremstyle{remark}

\numberwithin{equation}{section}
\numberwithin{figure}{section}

 \newcommand{\R}{{\mathbb R}}

\newcommand{\sph}{{\mathbb S}}

 \newcommand{\Cont}{{\mathcal C}}

\newcommand{\Tt}{{\EuScript T}}

\newcommand{\pol}{{\EuScript K}}

\newcommand{\Ll}{{\EuScript L}}

\newcommand{\Ww}{{\EuScript W}}
\newcommand{\Vv}{{\EuScript V}}

\newcommand{\cl}{\operatorname{Cl}}
\newcommand{\dist}{\operatorname{dist}}

\newcommand{\diam}{\operatorname{diam}}
\newcommand{\Conv}{\operatorname{Conv}}
\newcommand{\st}{\operatorname{St}}

\newcommand{\s}{{\tt s}}

\newcommand{\veps}{\varepsilon}
\newcommand{\eps}{\epsilon}
\newcommand{\ol}{\overline}

\usepackage{scalerel}

\usepackage{accents}

\newcommand\ov[1]{\accentset{\circ}{#1}} %Simplessi aperti

\begin{document}

\vspace{-2em}
\title[Differentiable approximation of continuous locally definable maps]{Differentiable approximation of continuous locally definable maps that preserves the image}


\date{\today}

\author{Antonio Carbone}
\address{Antonio Carbone, Dipartimento di Scienze dell'Ambiente e della Prevenzione, Palazzo Turchi di Bagno, C.so Ercole I D'Este, 32, Università di Ferrara, 44121 Ferrara (ITALY)}
\thanks{The author is partially supported by GNSAGA of INdAM}
\email{antonio.carbone@unife.it}

\begin{abstract}
Recently, we showed that continuous definable maps defined on compact definable sets can be uniformly approximated by continuous definable maps of class $\Cont^p$ without changing their image. The aim of this paper is to extend the previous result, this time taking into account the (strong) Whitney topology, to continuous locally definable maps defined on locally compact locally definable sets. The argument is an interplay between o-minimal and PL geometry and makes essential use of Paw\l{}ucki's desingularization techniques as well as our aforementioned result for the compact case.
\end{abstract}

\maketitle 

\section{Introduction}

Approximation is a tool of great importance in geometry. The importance of approximation of continuous maps is a natural question that arises from the classical Weierstrass approximation theorem: \textit{Every continuous map $f:X\to\R^m$, defined on a compact subset $X$ of $\R^n$, can be uniformly approximated on $X$ by polynomial maps.} Here, the approximating maps take values in the Euclidean space $\R^m$, but several difficulties arise when one tries to restrict the image of the approximating map to a fixed target space $Y\subset\R^m$. For instance, there exist no nonconstant polynomial maps from the sphere $\sph^2$ to the circle $\sph^1$. To overcome these difficulties one can consider, instead of polynomials, a more flexible class of approximating maps like differentiable maps. 

Let $X\subset \R^n$ and $Y\subset \R^m$ be any subsets and $p\geq 1$ an integer. Recall that a map $f:X\to Y$ is \textit{of class $\Cont^p$} if there exist an open neighbourhood $\Omega\subset \R^n$ of $X$ and an extension $F:\Omega\to \R^m$ of $f$ which is of class $\Cont^p$ in the usual sense. A classical example of differentiable approximation concerns Whitney's approximation theorem \cite{w} for continuous maps whose target space is a $\Cont^p$ submanifold of $\R^m$. Here, a crucial point is the existence of $\Cont^p$ tubular neighbourhoods for $\Cont^p$ submanifolds in $\R^m$. When the target space $Y$ has `singularities', the lack of existence of tubular neighbourhoods (of class $\Cont^p$ for $p\geq 1$), makes the approximation problem difficult to approach in its full generality. One possible way to proceed is to consider domains of definitions and target spaces in some suitable tame category. For instance, Fernando and Ghiloni \cite{fgh1,fgh2} conducted an extended study on differentiable approximation for continuous maps when the target space $Y$ admits suitable triangulations. They showed that differentiable approximation is possible for a wide class of triangulable sets including differentiable manifolds, polyhedra, semialgebraic sets, subanalytic sets and more generally definable sets of an o-minimal structure. We refer the reader to \cite[\S 1.1]{cf2} for an overview of the state of the art on the approximation problem in real algebraic geometry, and to \cite[\S 1]{ca} for a presentation of the problem more closely related to the topics discussed here. 

Recall that an \textit{o-minmal structure (on the ordered field of real numbers $\R$)} is a collection $\mathfrak{S}:=\{\mathfrak{S}_n\}_{n\in\mathbb{N}^*}$ of families of subsets of $\R^n$ satisfying the following properties:
\begin{itemize}
\item $\mathfrak{S}_n$ is a Boolean algebra,
\item $\mathfrak{S}_n$ contains the algebraic subsets of $\R^n$,
\item if $X\in \mathfrak{S}_n$ and $Y\in\mathfrak{S}_m$, then $X\times Y\in\mathfrak{S}_{n+m}$,
\item if $\pi:\R^n\times\R\to\R^n$ is the projection onto the first factor and $X\in\mathfrak{S}_{n+1}$, then $\pi(X)\in\mathfrak{S}_n$,
\item $\mathfrak{S}_1$ consists exactly of all the finite unions of points and intervals (of any type). 
\end{itemize}
The elements of $\mathfrak{S}_n$ are called \textit{definable subsets of} $\R^n$. A map $f:X\to Y$, between definable sets $X\subset\R^n$ and $Y\subset \R^m$, is a \textit{definable map} if its graph $\Gamma_f$ is a definable subset of $\R^{n+m}$. Let $p\geq 1$ be an integer. A definable map $f:X\to Y$ is a \textit{definable map of class $\Cont^p$} if there exist an open definable neighbourhood $\Omega\subset \R^n$ of $X$ and a definable extension $F:\Omega\to \R^m$ of $f$ which is of class $\Cont^p$ in the usual sense. As a consequence of the Tarski-Seidenberg theorem \cite[Thm.2.2.1]{BCR}, semialgebraic sets constitute an o-minimal structure, which is the `smallest' o-minimal structure, in the sense that it is contained in any other o-minimal structure. In particular, semialgebraic sets and maps are definable in any o-minimal structure. The collection of global subanalytic sets is precisely the collection of definable sets in the o-minimal structure $\R_{\text{an}}$ (see \cite{wi}). We refer the reader to \cite{vdD, vDM} for further information on the theory of o-minimal structures.

Let $\Omega\subset \R^n$ be an open set. A subset $X\subset \Omega$ is \textit{locally definable in $\Omega$} if for each $x\in \Omega$ there exists an open neighbourhood $U\subset\Omega$ of $x$ such that $X\cap U$ is a definable set. If $X$ is locally compact, then there exists an open neighbourhood $\Omega\subset \R^n$ of $X$ in which $X$ is close. As $\Omega\setminus X$ is open, then $X$ is locally definable in $\Omega$ if and only if for each $x\in X$ there exists an open neighbourhood $U\subset \R^n$ of $x$ such that $X\cap U$ is a definable set. As we will only deal with locally compact sets, in what follows, we will simply say that $X\subset \R^n$ is locally definable, without specifying the open set $\Omega$. In particular, for a locally compact set $X\subset \R^n$, locally definable means \textit{locally definable in any open set $\Omega\subset \R^n$ that contains $X$ as a closed subset.}

If $X\subset \R^n$ is an open subset, then it is locally semialgebraic (being a union of open balls), so, in particular, definable in any o-minimal structure. While, if $X\subset \R^n$ is locally definable and compact, then $X$ is definable, because it has a finite covering made of definable sets. A continuous map $f:X\to Y$ between locally compact locally definable subsets $X\subset \R^n$ and $Y\subset \R^m$ is \textit{locally definable} if its graph $\Gamma_f$, which is a locally compact subset of $\R^{n+m}$, is locally definable. Let $p\geq 1$ be an integer. A locally definable map $f:X\to Y$ is a \textit{locally definable map of class $\Cont^p$} if there exist an open neighbourhood $\Omega\subset \R^n$ of $X$ and a locally definable extension $F:\Omega\to \R^m$ of $f$ which is of class $\Cont^p$ in the usual sense. Observe that, in general, the image of a continuous locally definable map is not  guaranteed to be locally definable. For instance, it is easy to see that there exists a continuous locally semialgebraic map $\R\to \R^2$ whose image is the set $Y:=\bigcup_{s\geq 1}^{\infty}\big\{(x,y)\in\R^2 : \big(x-\tfrac{1}{s}\big)^2+y^2=\tfrac{1}{s^2}\big\}$, which is not locally definable at the origin.

For the rest of this article, even if not explicitly mentioned, when we refer to (locally) definable sets or (locally) definable maps we mean \textit{(locally) definable in a fixed o-minimal structure on the ordered field of real numbers $\R$.} 

\subsection{Presentation of the main result}

We denote by $\|\cdot\|_n$ the Euclidean norm of $\R^n$. Combining the results of Fernando and Ghiloni \cite{fgh1,fgh2} and the ones of Paw\l{}ucki \cite{Pa,pa2} one deduce straightforwardly the following:

\begin{thm}[{\cite[Thm.1.6]{pa2}}]\label{pafeghi}
Let $X\subset \R^n$ and $Y\subset \R^m$ be locally definable locally compact sets and $f:X\to Y$ a continuous map. Let $\veps:X\to \R$ be a strictly positive continuous function and $p\geq 1$ an integer. Then, there exists a map $g:X\to Y$ of class $\Cont^p$ such that $\|f(x)-g(x)\|_m<\veps(x)$ for each $x\in X$.
\end{thm}

Without additional assumptions on the involved map $f$ one cannot guarantee, in general, that the image of $f$ can be preserved after the approximation (i.e. $g(X)=f(X)$). For instance, if $f:\R\to \R^2$ is a surjective continuous map, one cannot find approximating maps $g$ of class $\Cont^p$, for $p\geq 1$, such that $g(\R)=\R^2$. In fact, as an application of Sard's theorem, there exist no surjective maps $g:\R\to \R^2$ of class $\Cont^p$ for $p\geq 1$.  

Our main result is the following. It ensure that if the involved continuos map $f$ is, in addition, locally definable and its image $f(X)$ is locally compact and locally definable, then it is always possible to find an approximating map $g$ of class $\Cont^p$ such that $g(X)=f(X)$. 

\begin{thm}\label{main1}
Let $X\subset \R^n$ and $Y\subset \R^m$ be locally compact locally definable sets. Let $\veps:X\to \R$ be a strictly positive continuous function and $p\geq 1$ an integer. Let $f:X\to Y$ be a continuous locally definable map such that $f(X)=Y$. Then, there exists a locally definable map $g:X\to Y$ of class $\Cont^p$ such that $\|f(x)-g(x)\|_m<\veps(x)$ for each $x\in X$ and $g(X)=Y$.
\end{thm}

The techniques already developed in \cite{ca} can be extended and used to prove the previous result under the additional assumption that the involved map $f$ is proper. However, they seem not adequate to deal with the general case, that is why, we have use a different strategy in our proof. The main difficulty is to guarantee that the approximating map $g$ preserves the image, that is $g(X)=Y$. The argument is an interplay between o-minimal and PL geometry and makes an essential use of (the brilliant) Paw\l{}ucki's desingularization techniques \cite{Pa, pa2} and of \cite[Thm.1.4]{ca}, which is the analogous of Theorem \ref{main1} in the compact case. 

\section{Preliminaries}

In this section we collect some preliminary concept and results that we will use freely in what follows.

\subsection{Whitney topology}

Let $X\subset\R^n$ and $Y\subset\R^m$ be locally compact sets. We denote by $\Cont^0(X,Y)$ the set of all continuous maps from $X$ to $Y$. We will consider on $\Cont^0(X,Y)$ the \textit{(strong) Whitney topology}. A fundamental system of open neighbourhoods of $f\in\Cont^0(X,Y)$ with respect to the Whitney topology is given by the sets 
$$
\mathcal{N}(f,\veps):=\{g\in\Cont^0(X,Y) : \|f(x)-g(x)\|_m<\veps(x) \text{\, for each\, } x\in X\},
$$
where $\veps:X\to \R$ is any strictly positive continuous function. In what follows we will make often use of the following well known result, whose proof follows by \cite[\S2.5, Ex.10]{hi} using standard arguments.

\begin{lem}
Let $X\subset \R^n$, $Y\subset \R^m$ and $Z\subset \R^p$ be locally compact set. Let $f:Y\to Z$ be a continuous map. Then, the map
$$
f_*:\Cont^0(X,Y)\to \Cont^0(X,Z), \quad h\mapsto f\circ h
$$
is continuous with respect to the Whitney topology.
\end{lem}

\subsection{Locally finite simplicial complexes}\label{precomplessi}

Given a subset $X\subset \R^n$ we denote by $\cl(X)$ the Euclidean closure of $X$ in $\R^n$. A \textit{simplex} $\sigma\subset\R^n$ of dimension $d$ is the convex hull of $d+1$ affinely independent points $\nu_0,\ldots,\nu_d\in\R^n$, that is
$$
\sigma=\Conv(\{\nu_0,\ldots,\nu_d\}):=\Big\{\lambda_0\nu_0+\ldots+\lambda_d\nu_d :  \lambda_0\geq 0,\ldots,\lambda_d\geq 0, \sum_{i=0}^d\lambda_i=1\Big\}.
$$
If $0\leq i_0<\ldots<i_k\leq d$, the simplex $\Conv(\{\nu_{i_0},\ldots,\nu_{i_k}\})$ is called a \textit{face of} $\sigma$ of dimension $k$. As usual, a face of dimension zero is called a \textit{vertex of} $\sigma$. We denote by $\partial\sigma$ the \textit{(relative) boundary of} $\sigma$ defined as the union of the proper faces of $\sigma$ and by $\ov{\sigma}:=\sigma\setminus\partial\sigma$ the \textit{(relative) interior of} $\sigma$, which is equal to the interior of $\sigma$ in the affine space generated by $\sigma$ in $\R^n$ (that is, the smallest affine subspace of $\R^n$ that contains $\sigma$). Observe that if $\sigma$ is a simplex of dimension zero, then $\partial \sigma=\varnothing$, so $\ov{\sigma}=\sigma$. Moreover, observe that $\sigma=\cl(\ov{\sigma})$ for each simplex $\sigma\subset \R^n$. 

A \textit{locally finite simplicial complex $\pol$ of $\R^n$} is a locally finite family of simplices of $\R^n$ such that 
\begin{itemize}
\item for each simplex $\sigma\in\pol$ all the faces of $\sigma$ belong to $\pol$,
\item for each $\sigma_1,\sigma_2\in\pol$ the intersection $\sigma_1\cap\sigma_2$ is either empty or a common face of both $\sigma_1$ and $\sigma_2$.
\end{itemize}
The \textit{(underlying) polyhedron of} $\pol$ is the set $|\pol|:=\bigcup\pol$ equipped with the topology induced by the Euclidean topology of $\R^n$. Observe that $|\pol|$ is a locally semialgebraic subset of $\R^n$, so locally definable in any o-minimal structure. A simplex $\sigma\in\pol$ is a \textit{maximal simplex of} $\pol$ if it is not a proper face of another simplex $\sigma'\in\pol$. We denote by $\pol_{\max}$ the set of all maximal simplices of $\pol$. A \textit{subcomplex $\Tt$ of} $\pol$ is a locally finite simplicial complex of $\R^n$ such that $\Tt\subset\pol$. Observe that for each simplex $\sigma\in\pol$, we have that $\sigma$ is the underlying polyhedron of the finite subcomplex $\{\sigma'\in\pol : \sigma'\subset \sigma\}$ of $\pol$. 

Let $\Tt$ be a subcomplex of $\pol$. The \textit{open star of $\Tt$ in} $\pol$ is the set 
$$
\st(|\Tt|,\pol):=\{\ov{\sigma}' : \sigma'\in\pol,\, \sigma'\cap\sigma=\varnothing \text{\, for some\, } \sigma\in \Tt\}.
$$
The set $|\st(|\Tt|,\pol)|:=\bigcup\st(|\Tt|,\pol)$ is an open neighbourhood of $|\Tt|$ in $|\pol|$, which is locally definable, because $\pol$ is locally finite. Moreover, if $\Tt$ is finite (for instance, if $|\Tt|$ is a simplex), then $|\st(|\Tt|,\pol)|$ is a definable set. The \textit{closed star of $\Tt$ in} $\pol$ is defined as 
$$
\ol{\st}(\sigma,\pol):=\{\sigma \in \pol: \sigma\subset \cl(|\st(\sigma,\pol)|)\}.
$$
The set $|\ol{\st}(|\Tt|,\pol)|:=\bigcup\ol{\st}(|\Tt|,\pol)$ is a locally compact locally definable neighbourhood of $|\Tt|$ in $|\pol|$. Moreover, $|\ol{\st}(|\Tt|,\pol)|=\cl(|\st(|\Tt|,\pol)|)$. Observe that if $\Tt$ is finite, then $|\ol{\st}(|\Tt|,\pol)|$ is a compact definable neighbourhood of $|\Tt|$ in $|\pol|$. The open star $\st(|\Tt|,\pol)$ is not a subcomplex of $\pol$ (it is a subcomplex if and only if $|\Tt|$ is an isolated vertex of $|\pol|$), while, the closed star $\ol{\st}(|\Tt|,\pol)$ is always a subcomplex of $\pol$. 

The \textit{second open star of $\Tt$ in} $\pol$ is the set 
$$
\st^{(2)}(|\Tt|,\pol):=\st(\ol{\st}(|\Tt|,\pol),\pol).
$$
The set $|\st^{(2)}(|\Tt|,\pol)|:=\bigcup\st^{(2)}(|\Tt|,\pol)$ is a locally definable open neighbourhood of $|\Tt|$ in $|\pol|$. The \textit{second closed star of $\Tt$ in} $\pol$ is defined as $\ol{\st}^{(2)}(|\Tt|,\pol):=\ol{\st}(\ol{\st}(|\Tt|,\pol),\pol)$. Observe that $\ol{\st}^{(2)}(|\Tt|,\pol)$ is a subcomplex of $\pol$. The set $|\ol{\st}^{(2)}(|\Tt|,\pol)|:=\bigcup\ol{\st}^{(2)}(|\Tt|,\pol)$ is a locally definable locally compact neighbourhood of $|\Tt|$ in $|\pol|$. Moreover, $|\ol{\st}^{(2)}(|\Tt|,\pol)|=\cl(|\st^{(2)}(|\Tt|,\pol)|)$. 

If $\pol$ is a locally finite subcomplex of $\R^n$, then $|\pol|$ is locally compact. In fact, given a point $x\in |\pol|$ there exists a simplex $\sigma\in\pol$ such that $x\in \sigma$ and $|\ol{\st}(\sigma,\pol)|$ is a compact neighbourhood of $x$ in $|\pol|$. If $\Tt$ is a locally finite subcomplex of $\pol$, then $|\Tt|$ is closed in $|\pol|$. In fact, let $x\in |\pol|\setminus|\Tt|$. Then, there exists a simplex $\sigma\in \pol$ such that $x\in \sigma$ and a vertex $\nu$ of $\sigma$ such that $\nu\not\in|\Tt|$. The set $|\pol|\setminus|\Tt|$ is open, because $|\st(\nu,\pol)|$ is an open neighbourhood of $x$ in $|\pol|\setminus|\Tt|$, so $|\Tt|$ is closed in $|\pol|$. In particular, given a subcomplex $\Tt$ of $\pol$, both the sets $|\st(|\Tt|,\pol)|$ and $|\st^{(2)}(|\Tt|,\pol)|$ are closed in $|\pol|$. 

 A \textit{(locally finite) refinement} (also called \textit{subdivision}) \textit{of} $\pol$ is a (locally finite) simplicial complex $\pol^*$ such that $|\pol^*|=|\pol|$ and each simplex $\sigma^*\in\pol^*$ is contained in some simplex $\sigma\in\pol$.

We end this section with the following:

\begin{lem}\label{piccolo}
Let $\Delta:\pol_{\max}\to \R,\, \sigma\mapsto \Delta(\sigma)$ be a strictly positive function. Then, there exists a strictly positive continuous function $\delta:|\pol|\to \R$ such that $\delta(x)<\Delta(\sigma)$ for each $\sigma\in\pol_{\max}$ such that $x\in \sigma$. 
\end{lem}
\begin{proof}
For each vertex $\nu$ of $\pol$ define
$$
\delta_\nu:=\frac{1}{2}\min\{\Delta(\sigma) : \sigma\in\pol_{\max}, \nu\in \sigma\},
$$
which is a well-define positive number, because $\pol$ is locally finite and each vertex $\nu$ belongs to at least one maximal simplex $\sigma\in\pol_{\max}$. Denote by $\Vv$ the set of all vertices of $\pol$. Let $\delta:|\pol|\to \R$ be the (unique) piecewise affine extension of the function $\Vv\to \R, \nu\mapsto \delta_\nu$. Clearly, $\delta$ is a continuous map. Let $x\in |\pol|$ and let $\sigma\in\pol_{\max}$ be a maximal simplex such that $x\in \sigma$, then
$
\delta(x)=\lambda_0\delta_{\nu_0}+\ldots+\lambda_s\delta_{\nu_s}
$
where $\nu_0,\ldots,\nu_s$ are the vertices of $\sigma$ and $(\lambda_0,\ldots,\lambda_s)$ the barycentric coordinates of $x$. As $\lambda_0\geq 0,\ldots,\lambda_s\geq 0$ and $\lambda_0+\ldots+\lambda_s=1$, we deduce that $\delta(x)>0$, because $\delta_\nu>0$ for each vertex $\nu\in\Vv$. Moreover,
$$
\delta(x)=\sum_{k=0}^s\lambda_k\delta_{\nu_k}=\frac{1}{2}\sum_{k=0}^s\lambda_k\min\{\Delta(\sigma') : \sigma'\in\pol_{\max}, \nu_k\in \sigma'\}\leq \frac{1}{2}\sum_{k=0}^s\lambda_k\Delta(\sigma)=\frac{1}{2}\Delta(\sigma)<\Delta(\sigma),
$$
as required.
\end{proof}

\section{Proof of Theorem \ref{main1}}

In this section we show Theorem \ref{main1}. We start with the following lemma. 

\begin{lem}\label{compconn}
Let $X\subset \R^n$ and $Y\subset\R^m$ be locally compact locally definable sets and $f:X\to Y$ be a locally definable map. Let $p\geq 1$ be an integer. If the restriction $f|_{X_i}:X_i\to Y$ is of class $\Cont^p$ for each connected component $X_i$ of $X$, then $f$ is a locally definable map of class $\Cont^p$.
\end{lem}
\begin{proof}
Let $\{X_i\}_{i\in I}$ be the family of connected component of $X$. As $X$ is a locally compact subset of $\R^n$, then it admits an exhaustion by compact sets. It follows, using standard arguments, that $I$ is at most countable. Let $\Omega\subset \R^n$ be an open neighbourhood of $X$ in $\R^n$ in which $X$ is closed. We claim: \textit{For each $i\in I$ there exists an open neighbourhood $\Omega_i$ of $X_i$ in $\Omega$ such that $\Omega_i\cap \Omega_j=\varnothing$ if $j\neq i$.} Fix $i\in I$. As $X$ is locally definable in $\Omega$, then for each $x\in X_i$ there exists an open definable neighbourhood $B_{i,x}$ of $x$ in $\Omega$ such that $B_{i,x}\cap X$ is definable, so, in particular, it has finitely many connected components \cite[Cor.3.3.6]{vdD}. Thus, $B_{i,x}\cap X_j\neq \varnothing$ only for finitely many indices $j\in I$, so, up to shrinking $B_{i,x}$ if necessary, we may assume that $B_{i,x}\cap X_j\neq \varnothing$ if and only if $j=i$. Define $\Omega_i:=\bigcup_{x\in X_i}B_{i,x}$ which is an open neighbourhood of $X_i$ in $\Omega$. Up to shrinking $\Omega_i$ if necessary, we may assume that $\cl(\Omega_i)\cap X_j=(\cl(\Omega_i)\cap \Omega)\cap X_j=\varnothing$ for each $j\neq i$. As $\Omega\setminus\cl(\Omega_i)$ is an open neighbourhood of $X\setminus X_i$ in $\Omega$ such that $(\Omega\setminus\cl(\Omega_i))\cap \Omega_i=\varnothing$, the claim follows easily using an inductive argument, because $I$ is at most countable. 

The restriction $f|_{X_i}$ is a locally definable map of class $\Cont^p$ for each $i\in I$. Thus, there exists an open neighbourhood $\Omega_i^*$ of $X_i$ in $\Omega_i$ and a locally definable extension $F_i:\Omega_i^*\to \R^m$ of $f|_{X_i}$ to $\Omega_i^*$. Let $\Omega^*:=\bigcup_{i\in I}\Omega^*_i$, which is an open neighbourhood of $X$ in $\R^n$. The map $F:\Omega^*\to \R^n$ defined as $F(x)=F_i(x)$ if $x\in \Omega_i^*$, is a well-defined locally definable extension of $F$ to $\Omega^*$ of class $\Cont^p$. We conclude that $f$ is a locally definable map of class $\Cont^p$, as required.
\end{proof}

The previous result is no longer true without the assumption that the involved map $f$ is locally definable, as shown in the following example. 

\begin{example}
Consider the compact set $X:=\{0\}\cup\{1/n : n\in \mathbb{N}^*\}$ and the function $f:X\to \R$ defined as $f(0)=0$ and $f(1/n):=n$ if $n\in\mathbb{N}^*$. The set $X$ is totally disconnected and the restriction of $f$ to each of its connected components is of class $\Cont^p$, because it is constant. The function $f$ is not locally bounded around $0$, so it does not admit any extension of class $\Cont^p$. In particular, $f$ is not of class $\Cont^p$. \hfill$\sqbullet$
\end{example}

We are ready to show Theorem \ref{main1}. We denote by $\dist(x,y):=\|x-y\|_n$ the (Euclidean) distance between the points $x,y\in\R^n$. Let $X,Y\subset \R^n$ be any subsets. We denote by $\dist(x,X):=\inf\{\dist(x,y) : y\in X\}$ the distance between the point $x\in \R^n$ and $X$. Recall that if $X$ is closed in $\R^n$ it holds, $\dist(x,X)=0$ if and only if $x\in X$. We denote by $\dist(X,Y):=\inf\{\dist(x,Y) : x\in X\}$ the distance between the sets $X$ and $Y$. If $X$ and $Y$ are compact, then $\dist(X,Y)=0$ if and only if $X\cap Y\neq\varnothing$. 

\begin{proof}[Proof of Theorem \ref{main1}]
The proof is long and technical and we divide it into several steps and subsequent reductions in order to help the reader to keep an overview on the argument. Let $\veps:X\to \R$ be a strictly positive continuous function and $p\geq1$ an integer.

%1
\noindent{\sc Step 1. Initial preparation.} As $Y\subset \R^m$ is a locally compact and locally definable subset of $\R^m$, by \cite[Thm.1.1]{pa2}, there exists a locally definable triangulation $(\Ll, \psi)$ of $Y$ in $\R^m$ such that $\psi:|\Ll|\to Y$ is a locally definable homeomorphism of class $\Cont^p$. The map 
$$
(\psi)_*:\Cont^0(X,|\Ll|)\to \Cont^0(X,Y), \quad g\mapsto \psi\circ g
$$ 
is continuous with respect to the Whitney topology. Thus, there exists a strictly positive continuous function $\delta_0:X\to \R$ such that if $g\in \Cont^0(X,|\Ll|)$ is such that $\|\psi^{-1}(f(x))-g(x)\|_m<\delta_0(x)$ for each $x\in X$, then 
$$
\|f(x)-\psi(g(x))\|_m=\|\psi(\psi^{-1}(f(x)))-\psi(g(x))\|_m<\veps(x)
$$
for each $x\in X$. In particular, up to substituting $f$ with $\psi^{-1}\circ f$ and $\veps$ with $\delta_0$, we may assume: \textit{$Y=|\Ll|$ is the (underlying) polyhedron of a locally finite simplicial complex $\Ll$ of $\R^m$.} Observe that $|\Ll|$ is still a locally compact locally definable subset of $\R^m$, because $\psi:|\Ll|\to Y$ is a locally definable homeomorphism.

As $X\subset\R^n$ is locally compact, then it is a locally closed subset of $\R^n$. Thus, it is contained as a closed subset of an open set $\Omega$ of $\R^n$. As $\Omega$ is open, then it is locally semialgebraic (being a union of open balls), so, in particular, locally definable. Let $(\pol,\varphi)$ be a locally definable triangulation of $\Omega$ in $\R^n$ subordinated to $X$. Denote by $\Tt$ the locally finite subcomplex of $\pol$ such that $|\Tt|=\varphi^{-1}(X)$. We may assume that $X$ is not compact, otherwise the statement follows by \cite[Thm,1.4]{ca}. In particular, the set of maximal simplices of $\Tt$ is countable, that is $\Tt_{\max}=\{\sigma_s\}_{s\geq 1}$. Here, we are implicitly assuming that $\sigma_s\neq \sigma_r$ if $s\neq r$.

Fix an integer $s\geq 1$. Define the set $Z_s:=\varphi(\sigma_s)$, which is a compact definable subset of $X$, because $\sigma_s$ is compact and $\varphi$ locally definable, so the restriction $\varphi|_{\sigma_s}$ is a definable map. Clearly, we have
\begin{equation}\label{unione}
X=\bigcup_{s\geq 1} Z_s,
\end{equation}
because $|\Tt|=\bigcup_{s\geq 1}\sigma_s$ and $X=\varphi(|\Tt|)$. Consider the set $R_s:=\varphi(|\ol{\st}(\sigma_s,\Tt)|)$, which is a compact definable neighbourhood of $Z_s$ in $X$. Observe that the family $\{R_s\}_{s\geq 1}$ is locally finite in $X$, because the family $\{|\ol{\st}(\sigma_s,\Tt)|\}_{s\geq 1}$ is locally finite in $|\Tt|$ and $\varphi|_{|\Tt|}:|\Tt|\to X$ is a homeomorphism. 

Let $s\geq 1$ be a fixed integer. The number of indices $r\geq 1$ such that $Z_r\cap Z_s=\varnothing$ if finite. In fact, $Z_r\cap Z_s\neq\varnothing$ if and only if $\sigma_r\cap \sigma_s\neq \varnothing$, because $\varphi$ is a homeomorphism, and this happens if and only if $\sigma_r\in\ol{\st}(\sigma_s,\Tt)$, because $\{\sigma_r\}_{r\geq 1}$ is the family $\Tt_{\max}$ of maximal simplices of $\Tt$. In particular, $Z_r\cap Z_s\neq\varnothing$ if and only if $Z_r\subset R_s$. Define the set 
\begin{equation}\label{Xi}
\Xi_s:=\{r\geq 1 : Z_r\cap Z_s\neq\varnothing\}=\{r\geq 1 : Z_r\subset R_s\}=\{r\geq 1 : \sigma_r\in\ol{\st}(\sigma_s,\Tt)\},
\end{equation}
where the last equality follows by the fact that $\varphi$ is a homeomorphism. Let $k_s$ be the (finite) cardinality of $\Xi_s$. Clearly, $k_s\geq 1$, because $s\in \Xi_s$. Define the set $R^{(2)}_s:=\varphi(|\ol{\st}^{(2)}(\sigma_s,\Tt)|)$, which is a compact definable subset of $X$. We claim: \textit{$R_s\subset R_r^{(2)}$ for each $r\in\Xi_s$.}

As $\varphi$ is a homeomorphism, it is enough to show that $|\ol{\st}(\sigma_s,\Tt)|\subset |\ol{\st}^{(2)}(\sigma_r,\Tt)|$ for each $r\in \Xi_s$. By \eqref{Xi}, $\Xi_s=\{r\geq 1 : \sigma_r\in\ol{\st}(\sigma_s,\Tt)\}$. Let $r\in \Xi_s$, as $\sigma_r\in\ol{\st}(\sigma_s,\Tt)$, then $\sigma_s\in\ol{\st}(\sigma_r,\Tt)$, so $|\ol{\st}(\sigma_s,\Tt)|\subset |\ol{\st}^{(2)}(\sigma_r,\Tt)|$, as claimed.

For each $s\geq 1$ define the integer $K_s:=\max\{k_r : r\in \Xi_s\}$. Observe that by the proof of the previous claim it follows also that for each $s\geq 1$ we have $r\in \Xi_s$ if and only if $s\in\Xi_r$. We deduce that 
$$
K_r=\max\{k_{\ell}: \ell\in\Xi_r\}\geq \max\{k_{\ell}: \ell\in\Xi_r\cap \Xi_s\}\geq k_s
$$
for each $r\in\Xi_s$, because $s\in\Xi_r$ for each $r\in\Xi_s$, so $s\in\Xi_r\cap \Xi_s$ for each $r\in\Xi_s$. In particular, $K_r\geq 1$, because $k_s\geq 1$. For each $s\geq 1$ define the number
\begin{equation}\label{vepss}
\veps_s:=\frac{1}{K_s}\min\{\veps(y) : y\in R^{(2)}_s\},
\end{equation}
which is well-defined and strictly positive, because $R^{(2)}_s$ is a compact set and $\veps$ a strictly positive continuous function. As $R_s\subset R_r^{(2)}$ for each $s\geq 1$ and $r\in\Xi_s$, we have that 
$$
\min\{\veps(y) : y\in R^{(2)}_r\}\leq\veps(x)
$$
for each $s\geq 1$, $r\in \Xi_s$ and $x\in R_s$. As $\tfrac{1}{K_r}\leq \tfrac{1}{k_s}$ for each $s\geq 1$ and each $r\in \Xi_s$, because $K_r\geq k_s$, we deduce that 
\begin{equation}\label{stimaVeps}
\sum_{r\in \Xi_s}\veps_r=\sum_{r\in \Xi_s}\frac{1}{K_r}\min\{\veps(y) : y\in R^{(2)}_r\}\leq \frac{1}{k_s}\sum_{r\in \Xi_s}\veps(x)=\frac{k_s}{k_s}\veps(x)=\veps(x)
\end{equation}
for each $s\geq 1$ and each $x\in R_s$.

%2
\noindent{\sc Step 2. Construction of a suitable family of weak retractions.} As $|\Ll|\subset\R^m$ is locally compact, then it is locally closed in $\R^m$. Thus, there exists an open neighbourhood $W\subset \R^m$ of $|\Ll|$ in which $|\Ll|$ is closed. By \cite[Prop.3.1, Rmk.3.2]{pa2}, there exists a locally definable triangulation $\Ww$ of $W$ such that $|\Ww|=W$ and $\Ll$ is a closed subcomplex of $\Ww$. We substitute $\Ww$ with its first barycentric subdivision, and consequently $\Ll$ with its first barycentric subdivision, so that $\Ll$ is still a closed subcomplex of $\Ww$. Let $\Ww^*:=\ol{\st}(|\Ll|,\Ww)$. Clearly, $\Ll$ is a closed subcomplex of $\Ww^*$. By \cite[Prop.8.3.3]{vdD} and its proof, there exists a locally definable retraction $\rho_0:|\Ww^*|\to |\Ll|$ such that for each $x\in |\Ww^*|$ the segment 
$$
\s_x:=\{(1-t)x+t\rho_0(x) : t\in [0,1]\}
$$ 
lies entirely in the simplex $\tau$ of $\Ww^*$ that contains the point $x$. Define $W^*:=|\st(|\Ll|,\Ww)|$, which is an open neighbourhood of $|\Ll|$ in $W=|\Ww|$, so an open neighbourhood of $|\Ll|$ in $\R^m$, because $W$ is open in $\R^m$. In order to lighten the notation, we substitute $\Ww$ with $\Ww^*$, so that $\Ww=\ol{\st}(|\Ll|,\Ww)$ and $W=|\st(|\Ll|,\Ww)|$. Observe that $|\Ww|\setminus W\neq \varnothing$.

Let $s\geq 1$ be a fixed integer. Let $\veps_s>0$ be the number introduced in \eqref{vepss}. By \cite[Thm.16.4]{mu}, there exists a locally finite subdivision $\Ww^*_s$ of $\Ww$ such that $\diam(\tau)<\tfrac{1}{3}\veps_s$ for each $\tau\in \Ww^*_s$. Clearly, $|\Ww^*_s|=|\Ww|$. By Paw\l{}ucki's desingularization \cite[Thm.1.1, Thm.1.2]{pa2}, there exists a locally definable triangulation $(\Ww_s,\psi_s)$ of $|\Ww|$ such that
\begin{itemize}
\item[\rm{(i)}] $\Ww_s$ is a locally finite refinement of $\Ww^*_s$,
\item[\rm{(ii)}] $\psi_s:|\Ww_s|=|\Ww|\to |\Ww|$ is of class $\Cont^p$,
\item[\rm{(iii)}] $\psi_s(\tau)=\tau$ for each $\tau\in \Ww^*_s$,
\item[\rm{(iv)}] $\rho_s:=\rho_0\circ\psi_s:|\Ww_s|=|\Ww|\to |\Ll|$ is of class $\Cont^p$.
\end{itemize}
As $\Ww^*_s$ is a refinement of $\Ww$, we have that each simplex $\tau$ of $\Ww$ is a union of simplices of $\Ww^*_s$. Thus, by (iii), we deduce also
 \begin{itemize}
\item[\rm{(v)}] $\psi_s(\tau)=\tau$ for each $\tau\in \Ww$.
\end{itemize}
As $\Ll$ is a subcomplex of $\Ww$, by (v), we have $\psi_s(x)\in |\Ll|$ for each $x\in |\Ll|$. As $\rho_0(x)=x$ for each $x\in |\Ll|$, because $\rho_0$ is a retraction onto $|\Ll|$, we deduce
 \begin{itemize}
\item[\rm{(vi)}] $\rho_s(x)=\rho_0(\psi_s(x))=\psi_s(x)$ for each $x\in |\Ll|$,
\end{itemize}
By (i), (iii), (vi) and the fact that $\diam(\tau)<\tfrac{1}{3}\veps_s$ for each $\tau\in \Ww^*_s$, we have
\begin{itemize}
\item[\rm{(vii)}] $\|x-\rho_s(x)\|_m=\|x-\psi_s(x)\|_m<\tfrac{1}{3}\veps_s$ for each $x\in |\Ll|$.
\end{itemize}

Let $x\in  |\Ww|$ and $\tau$ a maximal simplex of $\Ww$ such that $x\in \tau$. As $\Ww=\ol{\st}(|\Ll|,\Ww)$, $\Ll$ is a subcomplex of $\Ww$ and $\tau$ is a maximal simplex of $\Ww$, then the intersection $\tau_0:=\tau\cap |\Ll|$ is a face of $\tau$. By (v), $\psi_s(x)\in \tau$ for each $s\geq 1$. As for each $s\geq 1$ the segment  
$$
\s_{\psi_s(x)}:=\{(1-t)\psi_s(x)+t\rho_0(\psi_s(x)) : t\in [0,1]\}=\{(1-t)\psi_s(x)+t\rho_s(x) : t\in [0,1]\}
$$ 
lies entirely in $\tau$ (because $\tau$ is a simplex of $\Ww$ that contains $\psi_s(x)$) and $\rho_0(\psi_s(x))\in |\Ll|$, we deduce that $\rho_s(x)=\rho_0(\psi_s(x))\in \tau\cap|\Ll|=\tau_0$ for each $s\geq 1$. In particular, for each $x\in |\Ww|$, there exists a simplex $\tau_0$ of $\Ll$ such that $\rho_s(x)\in \tau_0$ for each $s\geq 1$. Thus, as the simplices of $\Ll$ are convex sets, we deduce the following crucial property:
\begin{itemize}
\item[\rm{(viii)}] for each choice of finitely many indices $s_1,\ldots,s_k\geq 1$ we have 
$$
t_1\rho_{r_1}(x)+\ldots+t_k\rho_{s_k}(x)\in|\Ll|
$$ 
for each $x\in |\Ww|$ and each $t_1\geq 0,\ldots,t_k\geq 0$ such that $t_1+\ldots+t_k=1$.
\end{itemize}

%3
\noindent{\sc Step 3. Weak retraction operator.} For each $s\geq 1$ let $Z_s$ and $R_s$ be the compact definable sets introduced in {\sc Step 1}. For each $s\geq 1$ let $U_s\subset \Omega$ be an open definable neighbourhood of $Z_s$ such that $U_s\cap X\subset R_s$. Observe that $\{U_s\cap X\}_{s\geq 1}$ is an open definable covering of $X$. This follows by \eqref{unione}, because $Z_s\subset U_s$ for each $s\geq 1$. Moreover, as $U_s\cap X\subset R_s$ for each $s\geq 1$, then the family $\{U_s\cap X\}_{s\geq 1}$ is locally finite in $X$, because the family $\{R_s\}_{s\geq 1}=\{\varphi(|\ol{\st}(\sigma_s,\Tt)|)\}_{s\geq 1}$ is locally finite in $X$. Observe that $U_s\cap X$ is definable for each $s\geq 1$, because $\cl(U_s)\cap X\subset R_s$ is compact, so definable. Let $\{\theta_s\}_{s\geq 1}$ be a partition of unity subordinated to the open covering $\{U_s\cap X\}_{s\geq 1}$ made of definable functions of class $\Cont^p$. For each $s\geq 1$ let $\rho_s:|\Ww|\to |\Ll|$ be the locally definable map of class $\Cont^p$ introduced in {\sc Step 2}, where $\Ww=\ol{\st}(|\Ll|,\Ww)$. Recall that $W=|\st(|\Ll|,\Ww)|\subset |\Ww|$. We define the operator
$$
\rho:\Cont^0(X, W)\to \Cont^0(X,\R^m), \quad h\mapsto \Big(x\mapsto \rho[h](x):=\sum_{x\in U_s\cap X}\theta_s(x)\rho_s(h(x))\Big).
$$
We start by showing:
\begin{itemize}
\item[\rm{(i)}] \textit{If $h\in \Cont^p(X,W)$ and $h$ is locally definable, then $\rho[h]\in \Cont^p(X,\R^m)$ and $\rho[h]$ is locally definable.}
\end{itemize} 

If $h\in\Cont^p(X,W)$, then $\rho[h]\in \Cont^p(X,\R^m)$, because in this case $\rho[h]$ is locally a finite sum of compositions of maps of class $\Cont^p$. For the same reason, if $h$ is locally definable, then $\rho[h]$ is also locally definable. Next, we show: 
\begin{itemize}
\item[\rm{(ii)}]  \textit{$\rho[h](X)\subset |\Ll|$ for each $h\in\Cont^0(X,W)$, that is, $\rho:\Cont^0(X, W)\to \Cont^0(X,|\Ll|)$.}
\end{itemize}

Let $h\in\Cont^0(X,W)$ and $x\in X$. As the covering $\{U_s\cap X\}_{s\geq 1}$ is locally finite, for each $x\in X$ there exist only finitely many indices $r_1,\ldots,r_k\geq 1$ such that $x\in U_{r_j}\cap X$, so $\theta_s(x)=0$ if $s\not\in\{r_1,\ldots,r_k\}$. Thus,
$$
\rho[h](x)=\theta_{r_1}(x)\rho_{r_1}(h(x))+\ldots+\theta_{r_k}(x)\rho_{r_k}(h(x)).
$$
As $h(x)\in W\subset |\Ww|$, $\theta_{r_1}(x)\geq 0,\ldots,\theta_{r_k}(x)\geq 0$ and $\theta_{r_1}(x)+\ldots+\theta_{r_k}(x)=1$, by property (viii) of {\sc Step 2}, we conclude that $\rho[h](x)\in |\Ll|$, as required. Then, we show: 
\begin{itemize}
\item[\rm{(iii)}]  \textit{Let $h\in\Cont^0(X,W)$ be such that $h(X)\subset |\Ll|$, then $\|h(x)-\rho[h](x)\|_m<\tfrac{1}{3}\veps(x)$ for each $x\in X$.} 
\end{itemize}

For each $s\geq 1$ let $\Xi_s$ be the set of indices introduced in \eqref{Xi}. Let $x\in X$. By \eqref{unione}, there exists $s\geq 1$ such that $x\in Z_s$. We claim: \textit{If $x\in U_r\cap X$, then $r\in\Xi_s$.} Let $r\geq 1$ be such that $x\in U_r\cap X$. Clearly, $U_r\cap Z_s=(U_r\cap X)\cap Z_s\neq \varnothing$, because $x\in U_r\cap Z_s$. As $U_r\cap X$ is open in $X$ and $\varphi$ is a homeomorphism, then the intersection $\varphi(U_r\cap Z_s)=\varphi(U_r)\cap \varphi(Z_s)=\varphi(U_r)\cap \sigma_r$ has non-empty interior in $|\Tt|$, because $\sigma_r$ is a maximal simplex of $\Tt$, so $\ov{\sigma}_r$ is open in $|\Tt|$. Thus, $|\ol{\st}(\sigma_r,\Tt)|\cap \sigma_s$ has non-empty interior in $|\Tt|$, because $\varphi(U_r)\cap \sigma_s=\varphi(U_r\cap X)\cap \sigma_s\subset |\ol{\st}(\sigma_r,\Tt)|\cap \sigma_s$. This implies that $\sigma_r\in\ol{\st}(\sigma_s,\Tt)$, because $\sigma_r$ is a maximal simplex of $\Tt$. We conclude that $r\in \Xi_s$, because, by \eqref{Xi}, $\Xi_s=\{r\geq 1 :\sigma_r\in\ol{\st}(\sigma_s,\Tt)\}$, as claimed. 

Let $x\in X$. By \eqref{unione}, there exists $s\geq 1$ such that $x\in Z_s$. By the previous claim, we deduce
\begin{equation}\label{buoniindici}
\theta_r(x)=0
\end{equation}
for each $r\geq 1$ such that $r\not\in \Xi_s$. As $h(X)\subset |\Ll|$, by (vii) in {\sc Step 2} and \eqref{stimaVeps}, we deduce
\begin{align*}
\|h(x)-\rho[h](x)\|_m&=\Big\|h(x)-\sum_{r\in \Xi_s}\theta_r(x)\rho_r(h(x))\Big\|_m=\Big\|\sum_{r\in \Xi_s}\theta_r(x)(h(x)-\rho_r(h(x))\Big\|_m\\
&\leq \sum_{r\in \Xi_s}\theta_r(x)\|h(x)-\rho_{r}(h(x))\|_m<\frac{1}{3}\sum_{r\in \Xi_s}\veps_r\leq \frac{1}{3}\veps(x),
\end{align*}
as required. Finally, we show:
\begin{itemize}
\item[\rm{(iv)}] \textit{For each $g\in \Cont^0(X,W)$ there exists a strictly positive continuos function $\delta:X\to \R$ such that if $h\in \Cont^0(X,W)$ is such that $\|g(x)-h(x)\|_m<\delta(x)$ for each $x\in X$, then 
$$
\|\rho[g](x)-\rho[h](x)\|_m<\frac{1}{3}\veps(x)
$$
for each $x\in X$.}
\end{itemize}

Observe that $\Cont^0(X,W)\subset \Cont^0(X,|\Ww|)$, because $W\subset |\Ww|$. Let $g\in \Cont^0(X,W)$. For each $s\geq 1$ the map $(\rho_s)_*:\Cont^0(X, |\Ww|)\to \Cont^0(X,|\Ll|)$ is continuous with respect to the Whitney topology. Thus, there exists a strictly positive continuous function $\delta^*_s:X\to\R$ such that if $h\in\Cont^0(X,|\Ww|)$ is a continuous function such that $\|g(x)-h(x)\|_m<\delta^*_s(x)$ for each $x\in X$, then $\|\rho_s(g(x))-\rho_s(h(x))\|_m<\tfrac{1}{3}\veps_s$ for each $x\in X$. Define $\delta_s^*:=\min\{\delta_s^*(x) : x\in  Z_s\}$, which is a well-defined strictly positive number for each $s\geq 1$, because $Z_s$ is compact. Recall that $\Tt_{\max}=\{\sigma_s\}_{s\geq 1}$. Define the function
$$
\Delta:\Tt_{\max}\to \R, \quad \sigma_s\mapsto \Delta(\sigma_s):=\min\{\delta^*_r : r\in\Xi_s\},
$$
which is strictly positive. By Lemma \ref{piccolo}, there exists a strictly positive continuous function $\delta^*:|\Tt|\to \R$ such that $\delta^*(y)<\Delta(\sigma_s)$ for each $\sigma_s\in\Tt_{\max}$ such that $y\in \sigma_s$. Define $\delta:X\to \R$ as $\delta:=\delta^*\circ\varphi^{-1}$, which is a strictly positive continuous function. Let $h\in \Cont^0(X,W)\subset \Cont^0(X,|\Ww|)$ be such that $\|g(x)-h(x)\|_m<\delta(x)$ for each $x\in X$. Let $x\in X$ and $s\geq 1$ be such that $x\in Z_s$ (we are using again \eqref{unione}), so $\varphi^{-1}(x)\in\varphi^{-1}(Z_s)=\sigma_s$. We deduce that,
$$
\|g(x)-h(x)\|_m<\delta(x)=\delta^*(\varphi^{-1}(x))<\Delta(\sigma_s)=\min\{\delta^*_r : r\in\Xi_s\}\leq \delta_r^*\leq \delta_r^*(x)
$$
for each $r\in \Xi_s$. Thus, $\|\rho_r(g(x))-\rho_{r}(h(x))\|_m<\tfrac{1}{3}\veps_r$ for each $r\in \Xi_s$. By \eqref{stimaVeps} and \eqref{buoniindici}, we conclude that
\begin{align*}
\|\rho[g](x)-\rho[h](x)\|_m&=\Big\|\sum_{r\in \Xi_s}\theta_r(x)\rho_r(g(x))-\sum_{r\in \Xi_s}\theta_r(x)\rho_r(h(x))\Big\|_m\\
&\leq \sum_{r\in\Xi_s}\theta_r(x)\|\rho_r(g(x))-\rho_{r}(h(x))\|_m<\frac{1}{3}\sum_{r\in \Xi_s}\veps_r\leq \frac{1}{3}\veps(x),
\end{align*}
as required.

It is worthwhile to remark that, even if the operator $\rho$ depends by the choice of the open covering $\{U_s\cap X\}_{s\geq 1}$ and the partition of unity $\{\theta_s\}_{s\geq 1}$, all the results of this step remain true for every choice of an open covering $\{U_s\cap X\}_{s\geq 1}$ made of definable sets and a partition of unity $\{\theta_s\}_{s\geq 1}$ made of definable functions of class $\Cont^p$, as long as $U_s\cap X\subset R_s$ for each $s\geq 1$.

%4
\noindent{\sc Step 4. Main reduction.} Let $\{U_s\cap X\}_{s\geq 1}$ be an open covering of $X$ made of definable sets and such that $U_s\cap X\subset R_s$ for each $s\geq 1$, and $\{\theta_s\}_{s\geq 1}$ a partition of unity subordinated to the open covering $\{U_s\cap X\}_{s\geq 1}$ made of definable functions of class $\Cont^p$. For each $s\geq 1$ let $\psi_s:|\Ww|\to|\Ww|$ be the locally definable homeomorphism of class $\Cont^p$ introduced in {\sc Step 2}. We define the map
\begin{equation}\label{f*}
f^*:X\to \R^m,\quad x\mapsto \sum_{x\in U_s\cap X}\theta_s(x)\psi_s^{-1}(f(x)).
\end{equation}
Clearly, the map $f^*$ is well-defined, because $f(X)\subset |\Ll|\subset |\Ww|$. Moreover, $f^*$ is continuous and locally definable, because locally it is a finite sum of compositions of continuous locally definable maps. We start by showing: 
\begin{itemize}
\item[\rm{(i)}] \textit{$f^*(X)\subset |\Ll|$.}
\end{itemize}

Let $x\in X$. By \eqref{unione}, there exists an integer $s\geq 1$ such that $x\in Z_s$. By \eqref{buoniindici}, we have
\begin{equation}\label{intermedio}
f^*(x)=\sum_{r\in\Xi_s}\theta_r(x)\psi_r^{-1}(f(x)).
\end{equation}
As $f(x)\in |\Ll|$ and $\Ll$ is a subcomplex of $\Ww$, then there exists a simplex $\tau_0\in \Ll\subset \Ww$ such that $f(x)\in \tau_0$. By (v) in {\sc Step 2}, $\psi_r(\tau)=\tau$ for each $r\geq 1$ and $\tau\in \Ww$, so $\psi_r^{-1}(\tau)=\tau$ for each $r\geq 1$ and $\tau\in\Ww$, because $\psi_r$ is a homeomorphism. In particular, $\psi_r^{-1}(f(x))\in \tau_0$ for each $r\in \Xi_s$. As $\tau_0$ is a convex subset of $\R^m$, by \eqref{intermedio}, we deduce that $f^*(x)\in |\Ll|$, as required. Next, we show:
\begin{itemize}
\item[\rm{(ii)}] \textit{$\|f(x)-f^*(x)\|<\tfrac{1}{3}\veps(x)$ for each $x\in X$.}
\end{itemize}

By (vii) in {\sc Step 2}, we have $\|y-\psi_r(y)\|_m<\tfrac{1}{3}\veps_r$ for each $y\in |\Ll|$ and each $r\geq 1$. As $\psi_r$ is a homeomorphism, then $\|y-\psi^{-1}_r(y)\|_m<\tfrac{1}{3}\veps_r$ for each $y\in |\Ll|$ and each $r\geq 1$. Let $x\in X$. As $f(X)\subset |\Ll|$, then $\|f(x)-\psi_r^{-1}(f(x))\|_m<\tfrac{1}{3}\veps_r$ for each $r\geq 1$. By \eqref{unione}, there exists an integer $s\geq 1$ such that $x\in Z_s$. By \eqref{stimaVeps} and \eqref{buoniindici}, we conclude that
\begin{align*}
\|f(x)-f^*(x)\|_m&=\Big\|\sum_{r\in \Xi_s}\theta_r(x)f(x)-\sum_{r\in \Xi_s}\theta_r(x)\psi^{-1}_r(f(x))\Big\|_m\\
&\leq\sum_{r\in \Xi_s}\theta_r(x)\|f(x)-\psi_r^{-1}(f(x))\|_m<\frac{1}{3}\sum_{r\in\Xi_s}\veps_r\leq\frac{1}{3}\veps(x),
\end{align*}
as required.

By (iv) in {\sc Step 3}, there exists a strictly positive continuous function $\delta:X\to \R$ such that if $h\in \Cont^{0}(X,W)$ is a continuous map such that $\|f^*(x)-h(x)\|_m<\delta(x)$ for each $x\in X$, then $\|\rho[f^*](x)-\rho[h](x)\|_m<\tfrac{1}{3}\veps(x)$ for each $x\in X$. Assume that $h:X\to W$ is a locally definable map of class $\Cont^p$ such that $\|f(x)-h(x)\|_m<\delta(x)$ for each $x\in X$. By (i), we have $f^*(X)\subset |\Ll|$. Thus, by (iii) in {\sc Step 3}, we have $\|f^*(x)-\rho[f^*](x)\|_m<\tfrac{1}{3}\veps(x)$ for each $x\in X$. By (ii), we deduce
\begin{align*}
\|f(x)-\rho[h](x)\|_m&\leq \|f(x)-f^*(x)\|_m+\|f^*(x)-\rho[f^*](x)\|_m+\|\rho[f^*](x)-\rho[h](x)\|_m\\
&<\frac{1}{3}\veps(x)+\frac{1}{3}\veps(x)+\frac{1}{3}\veps(x)=\veps(x)
\end{align*}
for each $x\in X$. By (i) in {\sc Step 4}, as $h:X\to W$ is a locally definable map of class $\Cont^p$, then $\rho[h]$ is also a locally definable map of class $\Cont^p$. By (ii) in {\sc Step 3}, it holds $\rho[h](X)\subset |\Ll|$. 

Thus, in order to conclude, we are reduced to show: \textit{There exists a locally definable map $h:X\to W$ of class $\Cont^p$ such that $\|f^*(x)-h(x)\|_m<\delta(x)$ for each $x\in X$ and $|\Ll|\subset\rho[h](X)$.}

%5
\noindent{\sc Step 5. Local approximation.} Let $W=|\st(|\Ll|,\Ww)|$ be the open neighbourhood of $|\Ll|$ in $\R^m$ introduced in {\sc Step 2}. Let $\partial W:=\cl(W)\setminus W$, which is a closed subset of $\R^m$, because $W$ is open. As $W=|\st(|\Ll|,\Ww)|$, then $|\Ww|=|\ol{\st}(|\Ll|,\Ww)|\subset \cl(W)$. Thus, $|\Ww|\setminus W\subset \cl(W)\setminus W=\partial W$. In particular, $\partial W\neq\varnothing$. Let $f^*$ be the map introduced in \eqref{f*}. By (ii) in {\sc Step 4}, we have $f^*(X)\subset |\Ll|$. Consider the function 
$$
\mu:X\to \R, \quad x\mapsto \dist(f^*(x),\partial W),
$$
which is continuous and strictly positive, because $f^*(x)\in |\Ll|\subset W$ for each $x\in X$ and $\partial W\neq \varnothing$ is a closed subset of $\R^m$ disjoint from $|\Ll|$. As $W\subset \R^m$ is open, then for each $y\in \R^m$ if $\|f^*(x)-y\|<\mu(x)$ for some $x\in X$, then $y\in W$. In particular, if $h:X\to \R^m$ is a continuous map such that $\|f^*(x)-h(x)\|_m<\mu(x)$ for each $x\in X$, then $h(X)\subset W$. Let $\delta:X\to \R$ be the strictly positive continuous function introduced in (iv) in {\sc Step 4}. In order to lighten the notation, we substitute $\delta$ with $\min\{\delta,\mu\}$, so that if $h:X\to \R^m$ is a continuous map such that $\|f^*(x)-h(x)\|_m<\delta(x)$ for each $x\in X$, then
\begin{equation}\label{stainW}
h(X)\subset W.
\end{equation}

We proceed analogously to {\sc Step 1} and for each $s\geq 1$ we define the number
$$
\delta_{1,s}:=\frac{1}{K_s}\min\{\delta(x) : x\in R^{(2)}_s\},
$$
which is well-defined and strictly positive, because $R^{(2)}_s$ is a compact set and $\delta$ a strictly positive continuous function. Arguing as in \eqref{stimaVeps} we find
\begin{equation}\label{stimaDelta}
\sum_{r\in \Xi_s}\delta_{1,r}\leq \delta(x)
\end{equation}
for each for each $s\geq 1$ and each $x\in R_s$. 

Let $s\geq 1$ be a fixed integer. As $Z_s$ is a compact definable set and the restriction $(\psi_s^{-1}\circ f)|_{Z_s}$ is a definable map, by \cite[Thm.1.4]{ca}, there exists a definable map $h_s:Z_s\to \R^m$ of class $\Cont^p$ such that $h_s(Z_s)=\psi_s^{-1}(f(Z_s))$ and 
\begin{equation}\label{stimaH1}
\|\psi_s^{-1}(f(x))-h_s(x)\|_m<\frac{1}{2}\delta_{1,s}
\end{equation}
for each $x\in Z_s$. As $Z_s$ is a closed subset of $\R^n$, by \cite[Lem.2.2]{cf3}, there exists an open definable neighbourhood $\Omega_s^*$ of $Z_s$ in $\Omega$ such that $\cl(\Omega_s^*)\subset \Omega$. It is worthwhile to remark that even if \cite[Lem.2.2]{cf3} is presented only for semialgebraic sets, its statement and proof work in the exact same way for definable sets. As $h_s:Z_s\to \R^m$ is definable and of class $\Cont^p$, then there exists an open definable neighbourhood $\Omega_s$ of $Z_s$ in $\Omega^*_s\subset \Omega$ and a definable extension $H_s:\cl(\Omega_s)\to \R^m$ of class $\Cont^p$ of $h_s$ to $\cl(\Omega_s)$ (we are implicitly using again \cite[Lem.2.2]{cf3}). Observe that $\Omega_s$ is an open definable neighbourhood of $Z_s$ in $\Omega$ such that $\cl(\Omega_s)\subset \cl(\Omega_s^*)\subset \Omega$. Up to shrinking $\Omega_s$ if necessary (and using again \cite[Lem.2.2]{cf3}), we may assume:
\begin{enumerate}[label=(\roman*)]
\item $\cl(\Omega_s)$ is compact. This is possible because $Z_s$ is a compact definable set.
\item $\cl(\Omega_s)\cap X\subset R_s$. This is possible because $R_s$ is a compact definable neighbourhood of $Z_s$ in $X$.
\item $\|\psi_s^{-1}(f(x))-H_s(x)\|_m<\tfrac{1}{2}\delta_{1,s}$ for each $x\in \cl(\Omega_s)\cap X$. This is possible because \eqref{stimaH1} holds on $Z_s$ and $\cl(\Omega_s)\cap X$ is a definable neighbourhood of $Z_s$ in $X$.
\end{enumerate}

%6
\noindent{\sc Step 6. Continuous `widening' of the sets $Z_s$.} In this step we want to suitably `widen' $Z_s$ inside $\Omega_s$. In order to do that, we take advantage of the triangulation $(\pol,\varphi)$ and its piecewise linear structure. Fix an integer $s\geq 1$. As $\cl(\Omega_s)$ is compact and $H_s:\cl(\Omega_s)\to \R^m$ continuous, then it is uniformly continuous. Thus, there exists $\delta_{2,s}>0$ such that for each $x,y\in \cl(\Omega_s)$ if $\|x-y\|_n<\delta_ {2,s}$, then 
\begin{equation}\label{stimaH2}
\|H_s(x)-H_s(y)\|_m<\frac{1}{2}\delta_{1,s}.
\end{equation}
As $\varphi:|\pol|\to \Omega$ is a locally definable homeomorphism and $\cl(\Omega_s)$ is a compact definable neighbourhood of $Z_s$ in $\Omega$, then $V_s:=\varphi^{-1}(\cl(\Omega_s))$ is a compact definable neighbourhood of $\sigma_s=\varphi^{-1}(Z_s)$ in $|\pol|$. As the restriction $\varphi|_{V_s}$ is continuous and $V_s$ is compact, then $\varphi|_{V_s}$ is uniformly continuous. In particular, there exists $\eps_s>0$ such that for each $x,y\in V_s$ if $\|x-y\|_n<\eps_{s}$, then 
\begin{equation}\label{phi}
\|\varphi(x)-\varphi(y)\|_n<\frac{1}{4}\delta_{2,s}.
\end{equation}

Let $\partial \Omega_s:=\cl(\Omega_s)\setminus \Omega_s$. As $\cl(\Omega_s)$ is a compact neighbourhood of the compact set $Z_s$, then
$$
\mu_s:=\frac{1}{2}\dist(Z_s, \partial \Omega_s):=\frac{1}{2}\min\{\dist(x,y) : x\in Z_s, y\in \partial\Omega_s\}
$$
is a well-defined strictly positive number. We substitute $\delta_{2,s}>0$ with $\min\{\delta_{2,s},\mu_s\}>0$. so that the set
\begin{equation}\label{insiemiEs}
E_s:=\{y\in \R^n : \|x-y\|_n<\delta_{2,s} \text{ for some } x\in Z_s\}
\end{equation}
is an open definable neighbourhood of $Z_s$ in $\Omega_s$, and, in particular, $E_s\subset \Omega_s$.

Recall that $\Tt_{\max}=\{\sigma_s\}_{s\geq 1}$. Let $\Pi_s$ be the affine space of $\R^n$ generated by $\sigma_s$ and $b_s$ the barycentre of $\sigma_s$. For each $\eps>1$ consider the homothety 
$$
\omega_{s,\eps}:\Pi_s\to \Pi_s, \quad x\mapsto b_s+\eps(x-b_s)
$$
of $\Pi_s$ of centre the barycentre $b_s$ of $\sigma_s$ and radius $\eps$. As $\sigma_s$ is compact and $V_s\cap \Pi_s$ a neighbourhood of $\sigma_s$ in $\Pi_s$, then there exists $\eps^*_s>1$ such that the homothety $\omega_{s,\eps^*_s}$ satisfies $\omega_{s,\eps^*_s}(\sigma_s)\subset V_s\cap \Pi_s$ and $\|x-\omega_{s,\eps^*_s}(x)\|_n<\eps_{s}$ for each $x\in \sigma_s$. By \eqref{phi}, $\|\varphi(x)-\varphi(\omega_{s,\eps^*_s}(x))\|_n<\tfrac{1}{4}\delta_{2,s}$ for each $x\in \sigma_s$. Thus, the map $\eta_s:=\varphi\circ\omega_{s,\eps^*_s}\circ\varphi^{-1}:Z_s\to \R^n$ satisfies
\begin{equation}\label{stimaeta1}
\|x-\eta_s(x)\|_n=\|\varphi(\varphi^{-1}(x))-\varphi(\omega_{s,\eps^*_s}(\varphi^{-1}(x)))\|_n<\frac{1}{4}\delta_{2,s}
\end{equation}
for each $x\in Z_s$, because $\varphi^{-1}(x)\in \sigma_s$. By \eqref{insiemiEs}, we have $\eta_s(Z_s)\subset E_s$, so $\eta_s(Z_s)\subset \Omega_s$, because $E_s\subset \Omega_s$.

\noindent{\sc Step 7. Smoothing of the `widening' to the class $\Cont^p$.} Let $s\geq 1$ be a fixed integer. Observe that $\sigma_s^*:=\omega_{s,\eps^*_s}^{-1}(\sigma_s)$ is contained in $\ov{\sigma}_s$, because $\eps_s>1$, so $\omega_{s,\eps^*_s}^{-1}$ is the homothety of centre $b_s$ and radius $0<\tfrac{1}{\eps^*_s}<1$. Define $Z_s^*:=\varphi(\sigma_s^*)$, which is a compact definable set strictly contained in $Z_s$. We show in this step: \textit{There exists an open definable neighbourhood $U_s$ of $Z_s$ in $\Omega_s$ and a definable map $\Gamma_s:U_s\to \Omega_s$ of class $\Cont^p$ such that $\Gamma_s(Z_s^*)=Z_s$ and $\|x-\Gamma_s(x)\|<\delta_{2,s}$ for each $x\in U_s$.}

As $\eta_s(Z_s^*)=\varphi(\omega_{s,\eps^*_s}(\varphi^{-1}(Z^*_s)))=\varphi(\omega_{s,\eps^*_s}(\varphi^{-1}(\varphi(\omega_{s,\eps^*_s}^{-1}(\varphi^{-1}(Z_s))))))=Z_s$ and $Z_s^*$ is a compact definable set, then, by \cite[Thm.1.4]{ca}, there exists a definable map $\gamma_{1,s}:Z^*_s\to Z_s$ of class $\Cont^p$ such that $\gamma_{1,s}(Z_s^*)=Z_s$ and $\|\eta_s(x)-\gamma_{1,s}(x)\|_n<\tfrac{1}{4}\delta_{2,s}$ for each $x\in Z_s^*$. By \eqref{stimaeta1} and the fact that $Z_s^*\subset Z_s$, we deduce that 
\begin{equation}\label{stimagamma1}
\|x-\gamma_s(x)\|_n\leq\|x-\eta_s(x)\|_n+\|\eta_s(x)-\gamma_s(x)\|_n<\frac{1}{4}\delta_{2,s}+\frac{1}{4}\delta_{2,s}=\frac{1}{2}\delta_{2,s}
\end{equation}
for each $x\in Z_s^*$. As the map $\gamma_{1,s}$ is definable and of class $C^p$, then there exist an open definable neighbourhood $U_{1,s}$ of $Z_s^*$ in $\Omega_s$ and a definable extension $\Gamma_{1,s}:U_{1,s}\to \R^n$ of class $\Cont^p$ of $\gamma_{1,s}$. Up to shrinking $U_{1,s}$ if necessary, we may assume that $\Gamma_{1,s}(U_{1,s})\subset \Omega_s$ (because $\Omega_s$ is an open definable neighbourhood of $Z_s^*$) and that $\|x-\Gamma_{1,s}(x)\|_n<\tfrac{1}{2}\delta_{2,s}$ for each $x\in U_{1,s}$ (because \eqref{stimagamma1} holds for each $x\in Z_s^*$). If $Z_s\subset U_{1,s}$, we conclude by taking $U_s:=U_{1,s}$ and $\Gamma_s:=\Gamma_{1,s}$.

In what follows we may assume that $Z_s\setminus U_{1,s}\neq \varnothing$, otherwise $Z_s\subset U_{1,s}$. Arguing as before, using again \cite[Thm.1.4]{ca}, we find a surjective definable map $\gamma_{2,s}:Z_s\to \eta_s(Z_s)\subset \Omega_s$ of class $\Cont^p$ such that $\|x-\gamma_{2,s}(x)\|_n<\tfrac{1}{4}\delta_{2,s}$ for each $x\in Z_s$. Moreover, we find an open definable neighbourhood $U_{2,s}$ of $Z_s$ in $\Omega_s$ and a definable extension $\Gamma_{2,s}:U_{2,s}\to \R^n$ of class $\Cont^p$ of $\gamma_{2,s}$ such that $\|x-\Gamma_{2,s}(x)\|_n<\tfrac{1}{2}\delta_{2,s}$ for each $x\in U_{2,s}$. 

As $Z_s^*$ is a proper subset of $Z_s$, then, up to shrinking $U_{1,s}$ if necessary, we may assume that $\cl(U_{1,s})\subset U_{2,s}$. By \cite[Lem.2.2]{cf3}, there exists an open definable neighbourhood  $U_{3,s}$ of $Z_s^*$ in $U_{1,s}$ such that $\cl(U_{3,s})\subset U_{1,s}$. Let $\{\lambda_1,\lambda_2\}$ be a definable partition of unity of class $\Cont^p$ subordinated to the open definable covering $\{U_{1,s},U_{2,s}\setminus \cl(U_{3,s})\}$ of $U_s:=U_{1,s}\cup U_{2,s}=U_{2,s}$. Define $\Gamma_s:=\lambda_1 \Gamma_{1,s}+\lambda_2\Gamma_{2,s}:U_s\to \R^n$, which is a definable map of class $\Cont^p$. We have
\begin{equation}\label{stimaGamma}
\|x-\Gamma_s(x)\|_n\leq \lambda_1(x)\|x-\Gamma_{1,s}(x)\|_n+\lambda_2(x)\|x-\Gamma_{2,s}(x)\|_n<\frac{1}{2}\delta_{2,s}+\frac{1}{2}\delta_{2,s}=\delta_{2,s}
\end{equation}
for each $x\in U_{s}$. Let $E_s$ be the set introduced in \eqref{insiemiEs}, which is an open definable neighbourhood of $Z_s$ in $\Omega_s$. By \eqref{stimaGamma}, we have that $\Gamma_s(U_s)\subset E_s$, so $\Gamma_s(U_s)\subset \Omega_s$, as required. 

\noindent{\sc Step 8. Suitable shrinking of the sets $U_s$.} Fix an integer $s\geq 1$. Recall that $Z_s=\varphi(\sigma_s)$, where $\sigma_s$ is a maximal simplex of $\Tt$ (see {\sc Step 1}). Let $Z_s^*$ and $U_s$ be the sets introduced in {\sc Step 7}. Recall that $Z_s^*$ is a suitable proper definable subset of $Z_s$ and that $U_s$ is a suitable open definable neighbourhood of $Z_s$ in $\Omega_s$. We show in this step: \textit{Up to shrinking $U_s$ if necessary, we may assume that $U_s\cap Z_r^*=\varnothing$ for each $r\neq s$.}

By the fact that $\sigma_r$ is a maximal simplex of $\Tt$ for each $r\geq 1$, we have that the relative interior $\ov{\sigma}_r$ of $\sigma_r$ in $|\Tt|$  satisfies $\sigma_s\cap\ov{\sigma}_r=\varnothing$ for each $r\neq s$. As $\varphi^{-1}(Z_r^*)\subset \ov{\sigma}_r$ for each $r\geq 1$, because $\varphi^{-1}(Z_r^*)=\sigma^*_r\subset \ov{\sigma}_r$ (see {\sc Step 7}), $\sigma_r=\varphi(Z_r)$ and $\varphi$ is a homeomorphism, we deduce that
\begin{equation}\label{disgiunti}
Z_s\cap Z^*_r=\varphi(\varphi^{-1}(Z_s))\cap \varphi(\varphi^{-1}(Z^*_r))=\varphi(\varphi^{-1}(Z_s)\cap \varphi^{-1}(Z^*_r))\subset \varphi(\sigma_s\cap \ov{\sigma}_r)=\varnothing.
\end{equation}
for each $r\neq s$. Recall that $R_s=\varphi(|\ol{\st}(\sigma_s,\Tt)|)$ and that $\Omega_s\cap X\subset R_s$ (see (ii) in {\sc Step 5}). Let $k_0:=s,k_1,\ldots,k_s$ be those integers such that $Z_{k_j}\subset R_s$ for $j=0,\ldots,s$. As $\Omega_s\cap X\subset R_s$ and $U_s\subset \Omega_s$, then $U_s\cap X\subset \Omega_s\cap X\subset R_s$. In particular, if $U_s\cap Z_r\neq \varnothing$, then $r\in\{k_0=s,k_1\ldots,k_s\}$. By \eqref{disgiunti}, we deduce
$$
Z_s\cap (Z^*_{k_1}\cup\ldots\cup Z^*_{k_s})=\varnothing.
$$
Thus, up to shrinking $U_s$ if necessary, we may assume that $U_s\cap Z_r^*=\varnothing$ for each $r\neq s$, as required.

\noindent{\sc Step 9. Global approximation.} Let $U_s$ be the definable sets introduced in {\sc Step 7}. By \eqref{unione} and by the fact that $Z_s\subset U_s$ for each $s\geq 1$, we have that $\{U_s\cap X\}_{s\geq 1}$ is an open covering of $X$ made of definable sets. Moreover, as $U_s\cap X\subset R_s$, then the family $\{U_s\cap X\}_{s\geq 1}$ is locally finite in $X$, because the family $\{R_s\}_{s\geq 1}=\{|\varphi(\ol{\st}(\sigma_s,\Tt))|\}_{s\geq 1}$ is locally finite in $X$. Let $\{\theta_s\}_{s\geq 1}$ be a partition of unity subordinated to the open covering $\{U_s\cap X\}_{s\geq 1}$ made of definable functions of class $\Cont^p$. For each $s\geq 1$ let $H_s:\cl(\Omega_s)\to \R^m$ be the locally definable map of class $\Cont^p$ introduced in {\sc Step 5} and $\Gamma_s:U_s\to \Omega_s$ the definable map of class $\Cont^p$ introduced in {\sc Step 7}. We define the map
$$
h:X\to \R^m, \quad x\mapsto \sum_{x\in U_s\cap X}\theta_s(x)H_s(\Gamma_s(x)).
$$
The map $h$ is locally definable, because locally it is a finite sum of definable maps. By Lemma \ref{compconn}, $h$ is also of class $\Cont^p$.

Let $f^*$ be the map introduced in \eqref{f*}. We show next: \textit{$\|f^*(x)-h(x)\|_m<\delta(x)$ for each $x\in X$.} For each $s\geq 1$ let $\Xi_s$ be the set of indices introduced in \eqref{Xi}. Let $x\in X$. By \eqref{unione}, there exists $s\geq 1$ such that $x\in Z_s$. As $U_r\cap X\subset \Omega_r\cap X\subset R_s$ for each $r\geq 1$ (see (ii) in {\sc Step 5} and recall that $U_r\subset \Omega_r$), by \eqref{buoniindici}, we have that $\theta_r(x)=0$ if $r\not\in \Xi_s$. We deduce that
$$
h(x)= \sum_{r\in \Xi_s}\theta_r(x)H_r(\Gamma_r(x))\quad \text{and}\quad f^*(x)=\sum_{r\in \Xi_s}\theta_r(x)\psi_r^{-1}(f(x)).
$$
As $U_r\cap X\subset \cl(\Omega_r)\cap X$ for each $r\in \Xi_s$, by (iii) in {\sc Step 5}, we have $\|\psi_r^{-1}(f(x))-H_r(x)\|_m<\tfrac{1}{2}\delta_{1,r}$ for each $r\in\Xi_s$. By \eqref{stimaH2} and \eqref{stimaGamma}, we have $\|H_r(x)-H_r(\Gamma_r(x))\|_m<\tfrac{1}{2}\delta_{1,r}$ for each $r\in \Xi_s$. Thus, by \eqref{stimaDelta}, we conclude that
\begin{align*}
\|f^*(x)-h(x)\|_m=&\Big\|  \sum_{x\in \Xi_s}\theta_r(x)\psi_r^{-1}(f(x))- \sum_{x\in \Xi_s}\theta_r(x)H_r(\Gamma_r(x))\Big\|_m\\
&\leq  \sum_{x\in \Xi_s}\theta_r(x)\|\psi_r^{-1}(f(x))-H_r(\Gamma_r(x))\|_m\\
&\leq  \sum_{x\in \Xi_s}\theta_r(x)\big(\|\psi_r^{-1}(f(x))-H_r(x)\|_m+\|H_r(x)-H_r(\Gamma_r(x))\|_m\big)\\
&< \sum_{x\in \Xi_s}\theta_r(x)\Big(\frac{1}{2}\delta_{1,r}+\frac{1}{2}\delta_{1,r}\Big)\leq \sum_{x\in \Xi_s}\delta_{1,r}\leq \delta(x),
\end{align*}
as required.

By \eqref{stainW}, as $\|f^*(x)-h(x)\|_m<\delta(x)$ for each $x\in X$, then $h(X)\subset W$. Thus, by {\sc Step 4}, in order to conclude our proof we are only left to show: $|\Ll|\subset \rho[h](X)$.

\noindent{\sc Step 10. Surjectivity of the approximating map.} Let $h:X\to W$ be the locally definable map of class $\Cont^p$ introduced in {\sc Step 9} and let 
$$
g:=\rho[h]:X\to |\Ll|, \quad x\mapsto \rho[h](x)=\sum_{x\in U_s\cap X}\theta_s(x)\rho_s\Big(\sum_{x\in U_s\cap X}\theta_s(x)H_s(\Gamma_s(x))\Big).
$$
It only remains to shows: $|\Ll|\subset g(X)$. 

Let $s\geq 1$ be a fixed integer. By {\sc Step 5}, we have $H_s(Z_s)=h_s(Z_s)=\psi_s^{-1}(f(Z_s))$. By {\sc Step 7}, we have $\Gamma_s(Z_s^*)=Z_s$. Let $\rho_0:|\Ww|\to |\Ll|$ be the locally definable retraction introduced in {\sc Step 2}. Clearly, $\rho_0(f(Z_s))=f(Z_s)$, because $f(Z_s)\subset |\Ll|$. Thus,
\begin{align}
\begin{split}\label{surget}
f(Z_s)&=\rho_0(f(Z_s))=\rho_0(\psi_s(\psi_s^{-1}(f(Z_s))))=\rho_0(\psi_s(H_s(Z_s)))\\
&=\rho_0(\psi_s(H_s(\Gamma_s(Z_s^*))))=\rho_s(H_s(\Gamma_s(Z_s^*))).
\end{split}
\end{align}
By {\sc Step 8}, we have $\theta_r(x)=0$ for each $x\in Z_s^*$ and each $r\neq s$, because $U_r\cap Z_s^*=\varnothing$ for each $r\neq s$. In particular, $\theta_s(x)=1$ for each $x\in Z_s^*$. Thus,
\begin{align*}
g(x)&=\rho[h](x)=\sum_{x\in U_r\cap X}\theta_r(x)\rho_r\Big(\sum_{x\in U_r\cap X}\theta_r(x)H_r(\Gamma_r(x))\Big)\\
&=\theta_s(x)\rho_s(\theta_s(x)H_s(\Gamma_s(x)))=\rho_s(H_s(\Gamma_s(x)))
\end{align*}
for each $x\in Z_s^*$. Then, $g(Z_s^*)=\rho_s(H_s(\Gamma_s(Z_s^*)))$. By \eqref{surget}, we deduce $g(Z_s^*)=f(Z_s)$. Thus, as $\bigcup_{s\geq 1}Z_s^*\subset \bigcup_{s\geq 1}Z_s=X$ and $f(X)=|\Ll|$, finally, we conclude that:
\begin{align*}
|\Ll|=f(X)=f\Big(\bigcup_{s\geq 1}Z_s\Big)=\bigcup_{s\geq 1}f(Z_s)=\bigcup_{s\geq 1}g(Z_s^*)=g\Big(\bigcup_{s\geq 1}Z^*_s\Big)\subset g(X),
\end{align*}
as required.
\end{proof}

\subsection*{Acknowledgments} The author would like to thank Enrico Savi for valuable discussions during the preparation of this work. The author would also like to thank Riccardo Ghiloni who suggested us to use Paw\l{}ucki's desingularization techniques to investigate differentiable surjective approximation.

\bibliographystyle{amsalpha}

\end{document}